\documentclass{article}
\usepackage{amssymb,amsthm}
\usepackage{amsmath}
\usepackage{amsfonts}
\usepackage{latexsym}
\usepackage{amssymb}
\usepackage[dvips]{graphics}
\newtheorem{theorem}{Theorem}

\newtheorem{lemma}[theorem]{Lemma}
\newtheorem{corollary}[theorem]{Corollary}
\newtheorem{definition}[theorem]{Definition}
\newtheorem{remark}[theorem]{Remark}
\newtheorem{example}[theorem]{Example}

\newcommand{\R}{{\mathbf R}}
\newcommand{\Z}{{\mathbf Z}}

\newcommand{\Q}{\mathbf Q}

\newcommand{\rk}{{\sf {rk}}}
\newcommand{\E}{{\mathbb E}}
\renewcommand{\P}{\mathbb P}

\newcommand{\sign}{{\rm {sign}}}

\begin{document}

\title{Topology of random 2-complexes}         
\author{A. Costa, M. Farber\footnote{Partly supported by a grant from the EPSRC.}\, and T. Kappeler\footnote{Partly supported by the Swiss National Science Foundation.}}        
\date{June 11, 2010}          
\maketitle

\abstract{We study the Linial--Meshulam model of random two-dimensional simplicial complexes. One of our main results states that for $p\ll n^{-1}$ a random 2-complex $Y$
collapses simplicially to a graph and, in particular, the fundamental group $\pi_1(Y)$ is free and $H_2(Y)=0$, a.a.s. We also prove that, 
 if the probability parameter $p$ satisfies $p\gg n^{-1/2+\epsilon}$, where $\epsilon>0$, then an arbitrary finite two-dimensional simplicial complex admits a topological 
embedding into a random 2-complex, with probability tending to one as $n\to \infty$. 
We also establish several related results, for example we show that for $p<c/n$ with $c<3$ the fundamental group of a random 2-complex contains a nonabelian free subgroup. 
Our method is based on exploiting explicit thresholds (established in the paper) for the existence of simplicial embedding and immersions of 2-complexes into a random 2-complex.}

\section{Introduction}

Modeling of large systems in applications motivates the development of unconventional geometric and topological notions. Among them are mixed probabilistic - topological concepts, 
such as the Erd\"os and R\'enyi random graphs of  \cite{ER}, which are currently used in many applications in engineering and computer science. 

More recently, higher dimensional analogs of the 
Erd\H{o}s-R\'enyi model were suggested and studied by Linial-Meshulam in~\cite{LM}, and 
Meshulam-Wallach in~\cite{MW}.
In these models one generates a random $d$-dimensional complex $Y$ by considering the full $d$-dimensional skeleton of the simplex 
$\Delta_n$ on vertices $\{1, \dots, n\}$ and retaining $d$-dimensional faces independently with probability $p$.

An interesting class of closed smooth manifolds depending on a large number of random parameters arise as configuration spaces of mechanical linkages with bars of random  lengths,
see  
\cite{F1}, \cite{FK}. Although the number of homeomorphism type of these manifolds grows extremely fast, their topological characteristics can be predicted with high probability when the number of links tends to infinity.

%

In this paper, we study the topology random two-dimensional complexes. 
The probability space $G(\Delta_n^{(2)}, p)$ of the Linial--Meshulam model of random 2-complexes is defined as follows. Let $\Delta_n$ denote the $(n-1)$-dimensional simplex 
with vertices $\{1, 2, \dots, n\}$. Then $G(\Delta_n^{(2)}, p)$ denotes the set of all 2-dimensional subcomplexes 
$$\Delta_n^{(1)}\subset Y\subset \Delta_n^{(2)},$$ containing the one-dimensional skeleton $\Delta_n^{(1)}$. The probability function 
$\P: G(\Delta_n^{(2)},p)\to \R$ is given by the formula
$$\P(Y) = p^{f(Y)}(1-p)^{{n\choose 3}-f(Y)}, \quad Y\in G(\Delta_n^{(2)}, p),$$
where $f(Y)$ denotes the number of faces in $Y$.  In other words, each of the 2-dimensional simplexes of $\Delta_n^{(2)}$ is included in a random 2-complex 
$Y$ with probability $p$, independently of the other 2-simplexes. As in the case of random graphs, $0<p<1$ is a probability parameter which may depend on $n$.  
The model $G(\Delta_n^{(2)}, p)$ includes all finite $2$-dimensional simplicial complexes containing the full 1-skeleton $\Delta_n^{(1)}$; however, the likelihood of various topological phenomena is dependent on the value of $p$. The theory of deterministic 2-complexes itself is a rich and active field of current research with many challenging open questions, see \cite{Hog}.

The fundamental group of a random 2-complex $Y\in G(\Delta_n^{(2)}, p)$ was investigated by Babson, Hoffman, and Kahle \cite{BHK}.  They showed that for 
$$p\gg n^{-1/2}\cdot (3\log n)^{1/2},$$ the group $\pi_1(Y)$ vanishes asymptotically almost surely (a.a.s)\footnote{We use the abbreviation a.a.s.~for the phrase \lq\lq asymptotically almost surely\rq\rq.}.
These authors use notions of negative curvature due to Gromov to study the nontriviality and hyperbolicity of $\pi_1(Y)$
for $$p\ll n^{-1/2-\epsilon}.$$

In \cite{CFK} it was shown that for $p\ll n^{-1-\epsilon}$, a random 2-complex $Y$ can be collapsed to a graph in $N$ steps, where 
$N=N(\epsilon)$ depends only on $\epsilon>0$. 

In this paper we prove the following theorem:

\begin{theorem} \label{thm1intro} If the probability parameter $p$ satisfies $$p\ll n^{-1}$$ then a random 2-complex $Y\in G(\Delta_n^{(2)},p)$ collapses simplicially to a graph, a.a.s. In particular,
the fundamental group $\pi_1(Y)$ is free and for any coefficient group $G$ one has $H_2(Y; G)=0$, a.a.s. 
\end{theorem}

We conjecture that a similar result holds for $d$-dimensional random complexes in the Meshulam - Wallach model \cite{MW}, i.e. for $p\ll n^{-1}$ a random $d$-dimensional complex collapses simplicially to a $(d-1)$-dimensional subcomplex. This would strengthen a theorem of D. Kozlov \cite{Ko}. 

Another major result of this paper states:

\begin{theorem}\label{thm2intro} 
Assume that for some $\epsilon>0$ the probability parameter $p$ satisfies $p\gg n^{-1/2+\epsilon}$. 
Let $S$ be an arbitrary simplicial finite 2-complex.
Then $S$ admits a topological embedding into a random 2-complex $Y\in G(\Delta_n^{(2)}, p)$, a.a.s. 
\end{theorem}

By {\it a topological embedding} $S\to Y$ we mean a simplicial embedding of a subdivision of $S$ into $Y$. 

The method of this paper (as well as the method of \cite{CFK}) is based on studying simplicial embeddings and immersions of polyhedra into random 2-complexes. 
We analyze in detail the numerical invariants $\mu(S)$ and $\tilde \mu(S)$, defined in section \S 3, which play a crucial role in the questions about the existence of embeddings and immersions. 
We also discuss the notion of balanced triangulations, a generalization of the notion of a balanced graph in the random graph theory. 
We prove that any triangulation of a closed surface is balanced although surfaces with boundary (even disks) admit unbalanced triangulations. 

Among some other results presented in this paper we may mention the statement that for $p<c/n$, where $c<3$, the fundamental group of a random 2-complex contains a nonabelian free subgroup, a.a.s. We also prove that for $p>c/n$ with $c>3$ the second homology group of a random 2-complex is nontrivial a.a.s; this strengthens a result of D. Kozlov \cite{Ko}. 

%

%
%
%

\newpage 
{\bf \noindent Basic Definitions} 
\vskip 0.3cm
For convenience of the reader we collect here the definitions of basic combinatorial notions related to 2-dimensional complexes which will be used in this paper.

Let $Y$ be a finite 2-dimensional simplicial complex. An edge of $Y$ is called {\it free} if it is included in exactly one 2-simplex.

The {\it boundary} $\partial Y$ is defined as the union of all free edges.  We say that a 2-complex $Y$ is {\it closed} if $\partial Y=\emptyset$. 

A $2$-complex 
$Y$ is called {\it pure} if every maximal simplex is 2-dimensional. By the {\it pure part} of a 2-complex we mean the maximal pure subcomplex, i.e. the union of all 2-simplexes.

Let $Y$ be a simplicial 2-complex and let $\sigma$ and $\tau$ be two 2-simplexes of $Y$. 
We say that $\sigma$ and $\tau$ are adjacent if they intersect in an edge.
The {\it distance} between $\sigma$ and $\tau$, $d_Y(\sigma, \tau)$, is the minimal integer $k$ such that there exists a sequence of 2-simplexes $\sigma=\sigma_0, \sigma_1, \dots, \sigma_k=\tau$ with the property that $\sigma_i$ is adjacent to $\sigma_{i+1}$ for every $0\le i<k$. (If no such sequence exists then $d_Y(\sigma, \tau)=\infty$.) The {\it diameter }
${\rm {diam}}(Y)$ is defined as the maximal value of $d_Y(\sigma, \tau)$ taken over pairs of 2-simplexes of $Y$. 

A simplicial 2-complex is {\it strongly connected} if it has a finite diameter. 


A {\it pseudo-surface} is a finite, pure, strongly connected 2-dimensional simplicial complex of degree at most $2$ (i.e., every edge is included in at most two 2-simplexes). 


\section{The fundamental group and the second Betti number}  
In this section we analyze the fundamental group and the second Betti number of a random 2-complex using mainly information provided by the Euler characteristic. The results of this section are specific for 2-dimensional random complexes.
\begin{theorem}\label{thm1} Suppose that $p<cn^{-1}$, where $c<3$. Then  
the fundamental group $\pi_1(Y)$ of a random 2-complex $Y\in G(\Delta_n^{(2)},p)$ contains a noncommutative free subgroup with probability at least $1-\lambda^{n^{2}},$ for all large enough $n$, 
where $$\lambda= \exp\left(-\frac{1}{8}\left(1-\frac{c}{3}\right)^2\right),$$ $0<\lambda<1$.  In particular, $\pi_1(Y)$ contains a free subgroup on two generators, a.a.s.
\end{theorem}

\begin{proof} The Euler characteristic of $Y\in G(\Delta_n^{(2)},p)$ can be written as
\begin{eqnarray}\label{Euler}\chi(Y) = n- {n\choose 2} +f_2(Y) = f_2(Y) +1 - {{n-1}\choose 2}
\end{eqnarray}
where $f_2(Y)$ denotes the number of $2$-simplexes in $Y$. Clearly, the function $f_2: G(\Delta_n^{(2)}, p)\to \Z$ coincides with the sum of random variables 
$$f_2 =\sum_\sigma I_\sigma$$ where
$\sigma$ runs over 2-simplexes $(i, j, k)$ (with $1\le i<j<k\le n$) and $I_\sigma(Y)=1$ iff $\sigma$ is included in $Y$; otherwise $I_\sigma(Y)=0$. Each $I_\sigma$ is a Bernoulli random variable with parameter $p$ and $f_2$ has binomial distribution
\begin{eqnarray*}\P(f_2(Y) =k) = {{n\choose 3}\choose k} p^k(1-p)^{{n\choose 3}-k}, \end{eqnarray*}
where $\quad k=0,1, 2, \dots, {n\choose 3}.$ The expectation $\E(f_2)$ equals $p{n\choose 3}$.
Using inequality (2.5) from \cite{JLR} we find that for any $t \ge 0$ 
\begin{eqnarray}\label{plus}
\P\left(f_2\ge  p{n \choose 3} +t\right) \le \exp\left(-\frac{t^2}{2(p {n\choose 3} +t/3)}\right).
\end{eqnarray}

Consider inequality (\ref{plus}) with
\begin{eqnarray}\label{3i} t= \left( 1- \frac{pn}{3}\right){{n-1}\choose 2} -1.\end{eqnarray}
We observe that: (i) the assumption $pn<c<3$ implies that $t>0$ for large $n$ and (ii) the inequality 
$f_2(Y) \ge p{n\choose 3}+t$ is equivalent to the inequality $\chi(Y) \ge 0$. 
We thus obtain from (\ref{plus})
\begin{eqnarray*}
\P\left( \chi(Y) \ge 0\right) \le \exp\left(-\frac{t^2}{2(p {n\choose 3} +t/3)}\right)
\end{eqnarray*}
and from (\ref{3i}), for $n \ge 3,$
\begin{eqnarray*}
p{n\choose 3} + \frac{t}{3} \le  \frac{1}{3}\left(\frac{2c}{3} + 1\right){{n-1}\choose 2} -\frac{1}{3} \le {{n-1}\choose 2}
\end{eqnarray*}
as $c < 3.$ Thus one gets for $n$ sufficiently large
\begin{eqnarray*}
\frac{t^2}{2(p {n\choose 3} +t/3)} & \ge & \frac{1}{2} \frac{[(1-\frac{pn}{3}){{n-1}\choose 2}-1]^2}{{{n-1}\choose 2}}\\
&\ge& \frac{1}{3} \left(1- \frac{c}{3}\right)^2\cdot {{n-1}\choose 2}\\
&\ge & \frac{1}{8} \left(1-\frac{c}{3}\right)^2 \cdot n^2.
\end{eqnarray*}
Therefore, by the definition of $\lambda$,
\begin{eqnarray*}
\P(\chi(Y) \ge 0) \le \exp\left(-\frac{1}{8}\left(1-\frac{c}{3}\right)^2 n^2\right) = \lambda^{n^{2}}
\end{eqnarray*}
and thus
\begin{eqnarray*} \P(\chi(Y) < 0) \ge 1 - \lambda^{n^{2}}.
\end{eqnarray*}

Theorem \ref{thm1} now follows from a theorem proven in \cite{FS} which states: 
{\it
If
the Euler characteristic of a finite connected two-dimensional polyhedron $Y$ is negative, $\chi(Y) < 0$, then $\pi_1(Y)$ contains a nonabelian
free subgroup.}

This completes the proof.

\end{proof}

\begin{theorem}\label{thm2} Suppose that $p>cn^{-1}$, where now $c>3$. Then for a random two-dimensional complex $Y\in G(\Delta_n^{(2)},p)$ one has
 $H_2(Y;\Z)\not=0$
with probability at least $1-\mu^{n^{2}},$ for all large enough $n$, where $$\mu
= \exp\left(-\frac{1}{8}\left(\frac{c}{3} -1\right)\right),$$ $0<\mu<1$.  In particular\footnote{Note that $H_2(Y;\Z)\not=0$ implies 
that $H_2(Y;G)\not=0$ for any coefficient group $G$.}, $H_2(Y;\Z)\not=0$, a.a.s.
\end{theorem}
\begin{proof} The proof is very similar to the one of Theorem \ref{thm1} and also uses the Euler characteristic. Clearly, $\chi(Y) = 1-b_1(Y)+b_2(Y)$ (where $b_i(Y)$ denotes the $i$-dimensional Betti number,
$b_i(Y)=\rk H_i(Y;\Z)$). Thus $\chi(Y)>1$ implies $b_2(Y)>0$. We will estimate from above the probability of the complementary event $\chi(Y)\le 1$. 

Using inequality (2.6) from \cite{JLR} one has for any $t \ge 0$
\begin{eqnarray*}\label{minus}
\P\left(f_2\le p{n \choose 3} -t\right) \le \exp\left(-\frac{t^2}{2p{n\choose 3}}\right).
\end{eqnarray*}
Now choose $$t= \left(\frac{pn}{3}-1\right)\cdot {{n-1}\choose 2}.$$ Since $pn>c>3$ we have 
$$t> \left(\frac{c}{3}-1\right)\cdot {{n-1}\choose 2}>0.$$ The inequality $f_2(Y)\le  p{n \choose 3} -t$ is equivalent to $\chi(Y) \le 1$. 
Thus we obtain
\begin{eqnarray*}
\P\left( \chi(Y) \le 1\right) \le \exp\left(-\frac{t^2}{2p{n\choose 3}}\right)
\end{eqnarray*}
and, for $n$ sufficiently large,
\begin{eqnarray*} \frac{t^2}{2p{n\choose 3}} &\ge &\frac{(\frac{pn}{3} - 1)^2 \cdot {{n-1}\choose 2}}{ 2 \frac{pn}{3}}\\
&\ge& \frac{1}{2} \left(\frac{pn}{3}-1\right) \cdot {{n-1}\choose 2}\\
&\ge & \frac{1}{8} \left(\frac{c}{3} - 1 \right) \cdot n^2.
\end{eqnarray*}

Finally, by the definition of $\mu,$ 
\begin{eqnarray*}
\P(b_2(Y)=0) \, \le\,  \P(\chi(Y) \le 1)\,  \le \, \mu^{n^{2}}.
\end{eqnarray*} 

This completes the proof.
\end{proof}

Next we consider the critical case $p=3/n$.

\begin{theorem} Assume that $p=\frac{3}{n}$. Then for any $\epsilon >0$ there exists $N$ such that for all $n>N$ the probability of each of the following statements (a) and (b) concerning a random 2-complex 
$Y\in G(\Delta_n^{(2)}, p)$ is greater than $\frac{1}{2} -\epsilon$: 

(a) the fundamental group $\pi_1(Y)$ contains a noncommutative free subgroup; 

(b) $H_2(Y;\Z)\not=0$. 
\end{theorem}

It is not known if  (a) and (b) exclude each other; one may ask about the probability that 
asymptotically, (a) and (b) hold simultaneously.

\begin{proof}
In the case when $p=3/n$ one has $\E(f_2)= {{n-1}\choose 2}$ and $\E(\chi)=1$ where 
 $f_2, \, \chi : G(\Delta_n^{(2)}, p)\to \Z$ are as above. From the De Moivre-Laplace Integral theorem \cite{Sh}, page 62, it follows that 
$$\P\left(f_2>{{n-1}\choose 2}\right) \sim \frac{1}{\sqrt{2\pi }}\int_0^\infty e^{-x^2/2}dx =\frac{1}{2}$$
and 
$$\P\left(f_2 \le {{n-1}\choose 2}-2\right) \sim \frac{1}{\sqrt{2\pi }}\int_{-\infty}^b e^{-x^2/2}dx \sim \frac{1}{2},$$
where $b$ can be found from the equation $b\sqrt{{n\choose 3}p(1-p)}=-2$, i.e. $$b=-\frac{2}{\sqrt{{ {n-1}\choose 2} \cdot \left(1-\frac{3}{n}\right) }}\sim 0.$$
By (\ref{Euler}), the inequality $f_2(Y) >{{n-1}\choose 2}$ is equivalent to $\chi(Y)>1$ and the inequality $f_2(Y)<{{n-1}\choose 2}-1$ is equivalent to $\chi(Y)< 0$. 
Thus we see that $$\P(\chi(Y)>1)\sim \frac{1}{2}, \quad \mbox{and}\quad \P(\chi(Y)<0) \sim \frac{1}{2}$$
and thus, for any given $\epsilon > 0,$
$$\P(\pi_1(Y)\supset F_2) \ge \P(\chi(Y)<0)\ge \frac{1}{2} - \epsilon,$$
$$\P(b_2(Y)>0) \ge \P(\chi(Y)>1)\ge \frac{1}{2} - \epsilon$$
for sufficiently large $n$. Here $F_2$ denotes the free group with two generators.
\end{proof}

\section{Simplicial embeddings and immersions}\label{secembeddings}

In this section we consider the containment problem for subcomplexes of random 2-dimensional complexes which is similar to the containment problem for random graphs, 
see \cite{JLR}, chapter 3. We also study simplicial immersions, which are more general than simplicial embeddings. 

Let $S$ be a 2-dimensional finite simplicial complex. We denote by $v=v_S$ and $f=f_S$ the numbers of vertices and faces of $S$ respectively. The set of vertices of $S$ is denoted by $V(S)$. We assume that $S$ is fixed, i.e. independent of $n$. 

\begin{definition}\label{def1}
A simplicial embedding $g: S\hookrightarrow Y$, where $Y\in G(\Delta_n^{(2)}, p)$ is a random 2-complex, is defined as an injective map of the set of vertices $V(S)$ of $S$ into the set of vertices $\{1, \dots, n\}$ of $Y$ satisfying the following condition: for any triple of distinct vertices $u_1, u_2, u_3\in V(S)$ which span a simplex in $S$, 
the corresponding points $g(u_1), g(u_2), g(u_3)\in \{1, \dots, n\}$ span a face of $Y$.   
\end{definition}

Next we define the following slightly more general notion. 

\begin{definition}\label{def2}
A simplicial immersion $g: S\looparrowright Y$ into a random 2-complex $Y\in G(\Delta_n^{(2)}, p)$  is defined as a map  of the set of vertices $V(S)$ of $S$ into the set of vertices $\{1, \dots, n\}$ of $Y$ satisfying the following two conditions: 

(a) 
for any triple of distinct vertices $u_1, u_2, u_3\in V(S)$ which span a $2$-simplex in $S$, 
the corresponding points $g(u_1), g(u_2), g(u_3)\in \{1, \dots, n\}$ are pairwise distinct and span a face of $Y$; 

(b) for any pair of distinct 2-simplexes $\sigma$ and $\sigma'$ of $S$, the corresponding 2-simplexes 
$g(\sigma)$ and 
$g(\sigma')$ of $Y$ are distinct.  
\end{definition}

Note that a simplicial immersion $g:S\looparrowright  Y$ is not necessarily injective on the set of vertices $V(S)$ but any pair of vertices $u_1, u_2\in V(S)$ with $g(u_1)=g(u_2)$ cannot lie in a 2-simplex of $S$. We also require that distinct 2-simplexes of $S$ are mapped to distinct 2-simplexes of $Y$.

If $g:S\looparrowright  Y$ is a simplicial immersion then for any subcomplex $S'\subset S$ the restriction $g|S'$ is also a simplicial immersion $S'\looparrowright  Y$.

\begin{lemma}\label{lm1}The probability that a 2-dimensional simplicial complex $S$ with $v$ vertices and $f$ faces admits a simplicial immersion into a random 2-complex 
$Y\in G(\Delta_n^{(2)}, p)$ is less or equal than $n^{v}p^{f}$, i.e.
\begin{eqnarray}
\P(S\looparrowright Y) \le n^vp^f.
\end{eqnarray}
\end{lemma}

\begin{proof} For a map $g: V(S) \to \{1, \dots, n\}$ denote by $J_g:G(\Delta_n^{(2)}, p)\to \{0,1\}$ the random variable such that $J_g(Y)=1$ if and only if 
$g$ determines a simplicial immersion $S\looparrowright  Y$, i.e. if the condition of Definition \ref{def2} is satisfies. 
Clearly, the expectation $\E(J_g)$ equals $p^{f}$.
%
The random variable 
$X_S=\sum_g J_g$ counts the number of simplicial immersions $S\looparrowright  Y$, where $g$ runs over all maps $V(S) \to \{1, \dots, n\}$. Thus
$$\E(X_S) =\sum_g \E(J_g) \le n^v\cdot p^f$$
and
$$\P(S\looparrowright  Y) = \P(X_S>0) \le \E(X_S) \le n^vp^f,$$
by the first moment method. 
\end{proof}
Next we define a useful numerical invariant which was also mentioned in \cite{BHK}. 
\begin{definition}\label{def6} For a simplicial 2-complex $S$ let 
$\mu(S)$
 denote 
$$\mu(S) 
=\frac{v}{f}\, \in \, \Q,$$
where $v=v_S$ and $f=f_S$ are the numbers of vertices and faces in $S$. 
\end{definition}

\begin{corollary}\label{cor7} If the probability parameter $p$ satisfies $$p\ll n^{-\mu(S)}$$ then the 2-complex $S$ admits no simplicial immersions into a  random 2-complex $Y\in G(\Delta_n^{(2)},p)$, a.a.s.
\end{corollary}

\begin{proof} The assumption $p\ll n^{-\mu(S)}$ means that $pn^{\mu(S)}\to 0$ as $n\to \infty$. Then $n^vp^f\to 0$ and the result now follows from Lemma \ref{lm1}. 
\end{proof}

As an example consider a simplicial graph $\Gamma$ and the cone over it $S=C(\Gamma)$. One has $v_S=v_\Gamma+1$ and $f_S=e_\Gamma$. Therefore 
$$\mu(S) = \frac{v_\Gamma+1}{e_\Gamma}.$$
Using Corollary \ref{cor7} we obtain:

\begin{corollary}\label{cor71} If a graph $\Gamma$ satisfies $\chi(\Gamma)<0$ then $\mu(S)\le 1$ where $S=C(\Gamma)$ is the cone over $\Gamma$. Therefore, if, $p\ll n^{-1}$, then the cone
$S=C(\Gamma)$ with $\chi(\Gamma)<0$ admits no simplicial immersions into a random 2-complex $Y\in G(\Delta_n^{(2)},p)$, a.a.s. 
\end{corollary}

This result will be used later in this paper.

\begin{definition}\label{def8} Let $S$ be a finite 2-dimensional simplicial complex. Define
\begin{eqnarray}
\tilde \mu(S) = \min_{S'\subset S} \mu(S'),
\end{eqnarray}
where the minimum is formed over all subcomplexes $S'\subset S$ or, equivalently, over all pure
subcomplexes $S' \subset S$. 
\end{definition}

Note that the invariant $\tilde \mu$ is monotone decreasing: if $S$ is a subcomplex of $T$ then $\tilde \mu(S) \ge \tilde \mu(T)$. 

The following result complements Corollary \ref{cor7}.

\begin{theorem}\label{embed} Let $S$ be a finite simplicial complex.  
\begin{enumerate}
  \item[(A)] If $p\ll n^{-\tilde \mu (S)}$ then the probability that $S$ admits a  simplicial immersion into a random 2-complex $Y\subset G(n, p)$ tends to zero as $n\to \infty$.
  \item[(B)] If $p\gg n^{-\tilde \mu (S)}$ then the probability that $S$ admits a simplicial embedding into a random 2-complex $Y\subset G(n, p)$ tends to one as $n\to \infty$. 
\end{enumerate}
\end{theorem}

\begin{proof} Let $S'\subset S$ be a subcomplex such that $\mu(S') = \tilde \mu(S) \le \mu(S)$. Then 
$$\P(S\looparrowright Y) \le \P(S'\looparrowright Y)$$ and $\P(S'\looparrowright Y)$ tends to zero assuming that $p \ll  n^{-\mu(S')}=n^{-\tilde \mu(S)}$ by Corollary \ref{cor7}. This proves the statement (A).

The following arguments prove the statement (B). 
Let $v$ denote the number of vertices of $S$. A simplicial embedding of $S$ into $Y$ is defined by an injective map 
$g:V(S) \to \{1, \dots, n\}$ where $V(S)$ is the set of vertices of $S$.
The function $X_S=\sum_{g}J_g: G(\Delta_n^{(2)},p)\to \Z$ counts the number of simplicial embeddings; here $g: V(S) \to \{1, \dots, n\}$ runs over all injective maps and $J_g$ denotes the random variable 
defined as in the proof of Lemma \ref{lm1}. 

For a pair of injective maps $g, g': V(S) \to \{1, \dots, n\}$ consider the pure subcomplex $H=H(g, g') \subset S$ which is defined as the union of all 2-simplexes $\sigma\subset S$ 
with the property $g(\sigma)\subset g'(S)$. 
Note that the product random variable 
$J_gJ_{g'}$ has the expectation 
$$\E(J_gJ_{g'})= p^{2f-f_H},$$
where $f=f_S$ is the number of faces of $S$ and $f_H$ is the number of 2-simplexes in $H$. 

Now we fix a pure subcomplex $H\subset S$ and
%
consider all ordered pairs of injective maps
$g, g': V(S) \to \{1, \dots, n\}$ with 
$H(g, g')=H$. 
The number $N$ of such pairs $g, g'$ satisfies $$ N \le C_{H} n^{2v - v_H}$$
for some constant  $C_H > 0 $ depending on $H$. 


The variance of $X_S$ can be estimated as follows
\begin{eqnarray*}
{\rm {Var}}(X_S) &=& \E(X_S^2)-\E(X_S)^2 \\ 
&=& \sum_{g,g'} \left[\E(J_gJ_{g'}) -\E(J_g)\E(J_{g'})\right]\\
& \le & \sum_{H\subset S} C_{H} n^{2v-v_H}\left[p^{2f-f_H}-p^{2f}\right]\\
&=& 
\sum_{H\subset S} C_{H} n^{2v-v_H}p^{2f-f_H}[1-p^{f_H}].
\end{eqnarray*}
Since for $n$ sufficiently large,
\begin{eqnarray*} 
\E(X_S) = {n\choose v}v! \cdot p^{f}\, \, \ge \, \frac{1}{2}\cdot n^{v} p^{f},\end{eqnarray*}
it follows that
$$\frac{{\rm {Var}}(X_S)}{ \E (X_S)^2} \le 4 \cdot (1-p) \cdot \sum_{H\subset S} (f_H C_H) \cdot (n^{v_H} p^{f_H})^{-1}.$$

Now, if $p\gg n^{-\tilde \mu(S)}$ then $n^{v_H}p^{f_H}\to \infty$ for any pure subcomplex  $H\subset S$ 
and therefore each term in the sum above tends to zero.
Thus using the Chebyshev inequality
$$\P(X_S=0)\le \frac{{\rm {Var}}(X_S)}{\E (X_S)^2},$$
we see that $\P(X_S=0)\to 0$ as $n\to \infty$. This implies statement (B).
\end{proof}

The above proof gives also the following quantitative statement:

\begin{corollary} Let $S$ be a fixed 2-complex. Then the probability $\P(S\not\subset Y)$ that $S$ is not embeddable into a random 2-complex $Y$ can be estimated by 
\begin{eqnarray}
\P(S\not\subset Y) \le C\cdot (1-p)\cdot \sum_{H\subset S, f_H>0} (n^{v_H}p^{f_H})^{-1},
\end{eqnarray}
where $C$ is a constant depending on $S$ and $H$ runs through all pure subcomplexes of $S$.  
\end{corollary}

\section{Proof of Theorem \ref{thm1intro}.}

In this section we prove Theorem \ref{thm1intro} stated in the Introduction. 

Note that the assumptions and conclusions of Theorem \ref{thm1intro} are stronger than those of Theorem \ref{thm1}.
One may also compare Theorem \ref{thm1intro} with the main result of \cite{CFK} which has stronger assumptions and conclusion than
Theorem \ref{thm1intro}. 


\begin{proof} For any triple of integers $x, y, z$ (with $x\ge 3$, $y\ge 3$ and $z\ge 0$) 
consider two graphs $\Gamma_{x,y,z}$ and $\Gamma'_{x,y,z}$ drawn schematically in Figure \ref{xyz}. 
\begin{figure}[h]
\begin{center}
\resizebox{12cm}{3cm}{\includegraphics[27,483][561,606]{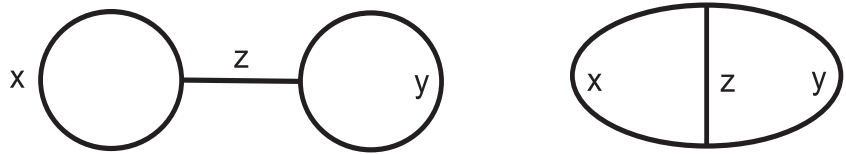}}
\end{center}
\caption{Graphs $\Gamma_{x,y, z}$ (left) and $\Gamma'_{x,y,z}$ (right).}\label{xyz}
\end{figure}
The graph $\Gamma_{x,y,z}$ is topologically the union of two circles joined by an interval; the circle on the left consists of $x$ intervals, the circle on the right is subdivided into
$y$ intervals, and the interval connecting them  consists of $z$ subintervals. The graph $\Gamma'_{x,y,z}$ shown schematically on the right of Figure \ref{xyz}, is the union of three arcs consisting of $x, y$ and $z$ intervals. Clearly $\chi(\Gamma_{x,y,z}) =-1=\chi(\Gamma'_{x,y,z})$. In the case $z=0$ the corresponding interval degenerates to a point.

It is easy to see that {\it any graph $\Gamma$ with $\chi(\Gamma)<0$ contains, as a subgraph, 
either $\Gamma_{x,y,z}$, or 
$\Gamma'_{x,y,z}$, for some $x,y,z$.}

Consider the cones $S_{x,y,z}=C(\Gamma_{x,y,z})$ and $S'_{x,y,z}=C(\Gamma'_{x,y,z})$. 
By the arguments leading to Corollary \ref{cor71} we have $$\mu(S_{x,y,z})=\mu(S'_{x,y,z})=1.$$
Applying Lemma \ref{lm1} we find
$$\P(S_{x,y,z}\looparrowright Y)\le (pn)^f$$ where $f=x+y+z$. Thus, 
\begin{eqnarray*}
\sum_{x,y\ge 3, \, z\ge 0} \P(S_{x,y,z}\looparrowright Y)&\le& \sum_{f\ge 6}f^2\cdot (pn)^f \\
&\le &\sum_{f\ge 6}(2pn)^f = \frac{(2pn)^6}{1-2pn}.
%
%
\end{eqnarray*}
We see that {\it if $pn\to 0$, then the probability that there exist $x,y,z$ such that the 2-complex $S_{x,y,z}$ admits a simplicial immersion into $Y$ tends to zero
 as $n\to \infty$.}

Similarly, {\it if $pn\to 0$, then the probability that there exist 
$x,y,z$ such that the 2-complex $S'_{x,y,z}$ admits a simplicial immersion into $Y$ tends to zero.}

Consider a vertex $v$ of the random 2-complex $Y$. The link $L_v$ of $v$ is a graph and the cone $C(L_v)$ embeds simplicially into $Y$. 
If for a connected component $L'_v$ of $L_v$ one has 
$\chi(L'_v)<0$ then 
for some integers $x,y,z$ the component $L'_v$ contains either $\Gamma_{x,y,z}$ or $\Gamma'_{x,y,z}$. 
Thus we see that $\chi(L'_v)<0$ implies that for some $x,y,z$ the complex $Y$ contains either $S_{x,y,z}$ or $S'_{x,y,z}$ as a subcomplex. 
Using the arguments given above we obtain that {\it for any vertex $v$ of $Y$, the Euler characteristic of 
every connected component 
$L'_v$ of the link $L_v$ of 
$v$ satisfies 
$$\chi(L'_v)\ge 0,$$ 
a.a.s.} In other words, every connected component of the link of any vertex of $Y$ is either contractible or is homotopy equivalent to the circle. 

Let $S$ be a pure and closed simplicial subcomplex of $Y$. The above arguments show that the link of any vertex of $S$ is a disjoint union of circles. In other words, we obtain that {\it any pure closed subcomplex $S\subset Y$ is a closed pseudo-surface}, i.e. every edge of $S$ is incident to exactly two 2-simplexes of $S$, a.a.s. 

For any two positive integers $x, y\ge 3$ with $\max(x, y) \ge 4$, let $L_{x, y}$ be a subdivision of the disk $D^2$ shown in Figure \ref{lxy}. 
The complex $L_{x,y}$ has two internal vertices $v$, $w$ such that the 
degree of $v$ is $x$ and the degree of $w$ is $y$. 
\begin{figure}[h]
\begin{center}
\resizebox{7cm}{4cm}{\includegraphics[67,417][438,678]{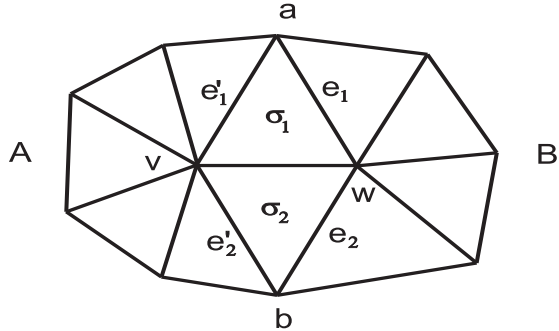}}
\end{center}
\caption{2-complex $L_{x,y}$.}\label{lxy}
\end{figure}
The total number of vertices of $L_{x,y}$ equals $x+y-2$; the number of faces of $L_{x,y}$ is also $x+y-2$; therefore $\mu(L_{x,y})=1$. 

In the special case $x=3$ and $y=3$ the complex $L_{3,3}$ is defined to be the tetrahedron with vertices $v, w, a, b$. 
%
%
The equality $\mu(L_{3,3})=1$ remains true.

By Lemma \ref{lm1}, $$\P(L_{x,y}\looparrowright Y)\le (pn)^f,$$ where $f=x+y-2$. Thus, 
\begin{eqnarray*}
\sum_{x, y\ge 3} \P(L_{x,y}\looparrowright Y)&\le& \sum_{f\ge 4}f\cdot (pn)^f \\
&\le & \sum_{f\ge 4} (2pn)^{f} = \frac{(2pn)^4}{1-(2pn)}.
\end{eqnarray*}
This shows that, if $pn\to 0$, then, with probability tending to one as $n\to \infty$, none of the complexes $L_{x,y}$ can be immersed $L_{x, y}\looparrowright Y$ into $Y$. 

Next we show that for any nonempty closed pseudo-surface $S$ there exist positive integers $x, y\ge 3$ and an immersion $L_{x,y}\looparrowright S$. 
Consider an edge $e=vw$ of $S$ and two 2-simplexes $\sigma_1$ and $\sigma_2$
incident to it, as shown on Figure \ref{lxy}. The link of $v$ in $S$ is a disjoint union of circles. It contains the edges $e_1$ and $e_2$ shown on Figure \ref{lxy}. Therefore we may find a simple arc $A$ in the link of $v$ in $S$ connecting the points $a$ and $b$ and disjoint from the interior of the arc $e_1\cup e_2$. 
 Similarly we may find a simple arc $B$ connecting $a$ and $b$ in the link of $w$ in $S$ and disjoint from the interior of 
$e'_1\cup e'_2$. 
Let $x$ and $y$ be such that the number of 2-simplexes in arc $A$ (correspondingly, $B$) is $x-2$ (correspondingly $y-2$). It is now obvious that we obtain {\it an immersion} of the 2-complex $L_{x,y}$ into $S$. It may not be an embedding since the images of some points of $A$ may coincide with the images of some points of $B$. 

Now we see that if a random 2-complex contains a closed pseudo-surface then there is an immersion $L_{x,y}\looparrowright Y$. 

Hence, summarizing the statements made above,  we conclude that {\it in the case $pn\to 0$, a random 2-complex 
contains no nonempty closed two-dimensional subcomplexes $S\subset Y$, a.a.s. }

Let $Y$ be a finite simplicial 2-complex. An edge of $Y$ is called {\it free} if it is incident to a single 2-simplex. A 2-simplex of $Y$ is called {\it free} if at least one of its edges is free. 
Let $\sigma_1, \dots, \sigma_k$ be all free 2-simplexes of $Y$; pick a sequence of free edges $e_1, \dots, e_k$ with $e_i\subset \sigma_i$. The subcomplex 
$$Y'=Y-\cup_{i=1}^k {\rm {int}}(\sigma_i) - \cup_{i=1}^k{\rm {int}}(e_i)$$
is obtained from $Y$ by collapsing all free 2-simplexes. The operation $Y\searrow Y'$ is called { \it a simplicial collapse}. Clearly, $Y'\subset Y$ is a deformation retract of $Y$. 

The procedure of collapse can be iterated $Y\searrow Y' \searrow Y''\searrow \dots$. There are two possibilities: either (a) after 
a finite number of collapses 
we obtain a {\it closed} 2-dimensional complex $Y^{(k)}$; or (b) for some $k$ the complex $Y^{(k)}$ is one-dimensional, i.e. a graph. 

Our discussion above implies that if $pn \to 0$, then for a random 2-complex $Y$ the possibility (a) happens with probability tending to $0$. Therefore, with probability tending to $1$, 
a random 2-complex collapses to a graph, under the assumption $p\ll n^{-1}$. 

This completes the proof.
\end{proof}

\begin{remark} {\rm The main step of the above proof was to show that for $p\ll n^{-1}$ a random 2-complex $Y$ contains no nonempty closed 2-dimensional subcomplexes 
$S\subset Y$. From Lemma 19 of 
\cite{CFK} we know that for any closed 2-complex $S$ one has $\tilde \mu(S) \le 1$. Therefore, given a closed 2-complex $S$, we may apply Theorem \ref{embed} to conclude that 
the probability that this $S$ embeds into a random 2-complex $Y\in G(\Delta_n^{(2)}, p)$ tends to zero as $n\to \infty$. However this would not be strong enough to prove Theorem 
\ref{thm1intro} since we need to know (as shown in the proof above) that the probability that there exists a closed 2-complex $S$ which embeds to a random 2-complex tends to zero. 

}
\end{remark}
\section{Surfaces in random 2-complexes}

In this section we apply the results of section \S \ref{secembeddings}
and study embeddings of triangulated surfaces into random 2-dimensional complexes.

\begin{definition}\label{balanced} A finite simplicial 2-complex $S$ is called balanced if $$\mu(S) = \tilde \mu(S),$$ i.e. if the quantities defined in Definitions \ref{def6} and \ref{def8} coincide. 
In other words, $S$ is balanced if $$\mu(S) \le \mu(S')$$ for any subcomplex $S'\subset S$.
\end{definition}

Definition \ref{balanced} is similar to the corresponding notion for random graphs, see \cite{JLR}. 

In this section we show that there exist many unbalanced triangulations of the disk however all closed triangulated surfaces are balanced. 
We start with the following observation.

\begin{lemma} A connected simplicial $2$-complex $S$ is balanced if and only if $\mu(S)\le \mu(S')$ for all connected subcomplexes $S'\subset S$. 
\end{lemma}
\begin{proof}Let $S'= S'_1\sqcup S'_2$ be a disjoint union of two subcomplexes. We show that 
$$\mu(S') \ge \min\{\mu(S'_1), \mu(S'_2)\}$$
and thus $\mu(S) \le \mu(S'_i)$, where $i=1, 2$, implies $\mu(S) \le \mu(S')$. 
Let $v_i$ and $f_i$ denote the number of vertices and faces of $S'_i$, $i=1, 2$. Assume that $v_1/f_1\le v_2/f_2$. Then one easily checks 
$$\mu(S') =\frac{v_1+v_2}{f_1+f_2} \ge \frac{v_1}{f_1}=\mu(S'_1).$$
The result now follows by induction on the number of connected components of $S'$. 
\end{proof}

\begin{example} {\rm Let $S=\Sigma_g$ be a triangulated closed orientable surface of genus $g\ge 0$. Then $\chi(S) = 2-2g= v-e+f$ where $v, e, f$ denote the numbers of vertices, edges and faces in $S$ correspondingly. Each edge is contained in two faces which gives $3f=2e$ and therefore 
\begin{eqnarray}\label{orient}
\mu(\Sigma_g) = \frac{1}{2} + \frac{2-2g}{f}.
\end{eqnarray}
Similarly, if $S=N_g$ is a triangulated closed nonorientable surface of genus $g\ge 1$ then $\chi(N_g) = 2-g$ and 
\begin{eqnarray}\label{nonorient}
\mu(N_g) = \frac{1}{2} + \frac{2-g}{f}.
\end{eqnarray}
}
\end{example}
Formulae (\ref{orient}) and (\ref{nonorient}) give the following:
\begin{corollary}\label{cor13} The invariants $\mu(\Sigma_g)$ of orientable triangulated surfaces satisfy:
\begin{enumerate}
  \item $1/2< \mu(\Sigma_g)\le 1$ for $g=0$ (since $f\ge 4$);
  \item $\mu(\Sigma_g)=1/2$ for $g=1$ (the torus);
  \item $\mu(\Sigma_g)<1/2$ for $g>1$;
  \item If $f\to \infty$ (i.e. when the surface is subsequently subdivided) then $\mu(\Sigma_g)\to 1/2$.\end{enumerate}
\end{corollary}

\begin{corollary}\label{cor14} The invariants $\mu(N_g)$ of nonorientable triangulated surfaces satisfy:
\begin{enumerate}
  \item $1/2< \mu(N_g)\le 3/5$ for $g=1$ (since $f\ge 10$);
  \item $\mu(N_g)=1/2$ for $g=2$ (the Klein bottle);
  \item $\mu(N_g)<1/2$ for $g>2$;
  \item If $f\to \infty$ (i.e. when the surface is subsequently subdivided) then $\mu(N_g)\to 1/2$.\end{enumerate}
\end{corollary}
Here we used the well-known fact that any triangulation of the real projective plane  ${\mathbf {RP}}^2$ has $f\ge 10$ faces, see \cite{Hea}, \cite{JR}, \cite{HR}, \cite{Rin}. 

\begin{example}\label{exa2} {\rm Let $S$ be a triangulated disc. Then $\chi(S)=v-e+f=1$ and $3f=2e-e_0$ where $e_0$ is the number of edges in the boundary $\partial S$. 
Substituting $e= (3f+e_0)/2$, one obtains
\begin{eqnarray}
\mu(S) = \frac{1}{2} + \frac{e_0}{2f} + \frac{1}{f}.
\end{eqnarray}

As a specific example consider the regular $n$-gon $S$ shown on Figure \ref{fig1} left. Then $v=n+1$, $f=n$ and $$\mu(S)=1+\frac{1}{n}.$$ 

On Figure \ref{fig1} 
on the right we have 
$e_0=4$ and the number of faces $f$ equals $f=2n+4$. Thus 
$$\mu(T) = \frac{1}{2} +\frac{3}{2n+4},$$
converges to $\frac{1}{2}$ as $n\to \infty$. 
}
\end{example}
\begin{corollary} For any triangulation $S$ of the disk one has $\mu(S) >1/2$. There exist triangulations $S$ of $D^2$ with $\mu(S)$ arbitrarily close to $1/2$. 
\end{corollary}

\begin{figure}[t]
\begin{center}
\resizebox{11cm}{4.5cm}{\includegraphics[21,424][529,626]{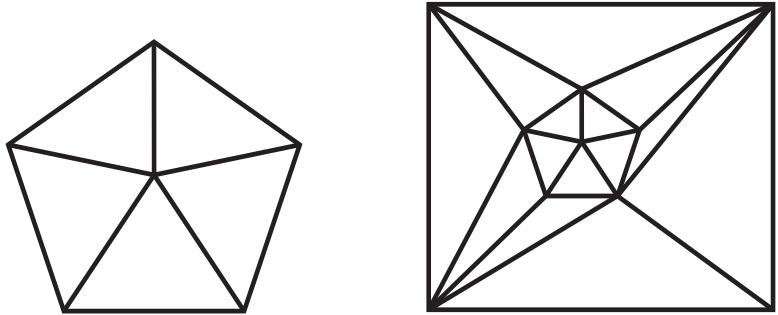}}
\end{center}
\caption{An $n$-gon $S$ (left) and a square with implanted $n$-gon $T$ (right).}\label{fig1}
\end{figure}

\begin{example}\label{exa3}{\rm Let $S'$ be such that $\mu(S')<1$ and suppose that $S$ is obtained from $S'$ by adding a triangle $\Delta$ such that $S'\cap \Delta$ is an edge. Then $S$ is not balanced. Indeed, $v_S = v_{S'} +1$ and $f_S=f_{S'}+1$ and $$\mu(S) =\frac{v_{S'}+1}{f_{S'}+1} > \frac{v_{S'}}{f_{S'}} = \mu(S').$$
}
\end{example} 

\begin{corollary} There exist unbalanced triangulations of the disk.
\end{corollary}
\begin{proof} Start with a disk triangulation $S'$ with $\mu(S')<1$ (for instance, $S'$ can be the square with implanted $n$-gon, see Example \ref{exa2}) and add a triangle $S=S'\cup \Delta$ such that $S'\cap \Delta$ is an edge lying in the boundary $\partial S'$. Then $\mu(S) >\mu(S')$ (see Example \ref{exa3}) and $S$ is unbalanced. Clearly, $S$ is homeomorphic to the $2$-dimensional disk.
\end{proof}

\begin{theorem}\label{bal} Any closed connected triangulated surface $S$ is balanced. 
\end{theorem}
\begin{proof} Let $S'\subset S$ be a connected subcomplex, $S'\not=S$. We may assume that each edge of $S'$ belongs to either one or two triangles of $S'$ (since any edge which is not incident to a triangle can be simply removed without affecting $\mu(S')$). Then we have
\begin{eqnarray}\label{euler}
\chi(S') = 1-b_1(S') = v'-e'+f'
\end{eqnarray}
where $v', e', f'$ are the numbers of vertices, edges and faces in $S'$. Here we use the assumption that $S'$ is connected (i.e. $b_0(S')=1$) and $S'\not=S$ (i.e. $b_2(S')=0$). 
One may write
$$3f'=2e'-e_0$$
where $e_0$ is the number of edges incident to exactly one 2-simplex. Expressing $e'$ through $f',$ and $e_0$ and substituting into (\ref{euler})
we obtain
\begin{eqnarray}\label{sprime}
\mu(S') = \frac{1}{2} + \frac{1-b_1(S')}{f'} + \frac{e_0}{2f'}.
\end{eqnarray}
Assume first that $S$ is orientable and has genus $g$, i.e. $S=\Sigma_g$. Then we have formula (\ref{orient}) and the inequality $\mu(S') \ge \mu(S)$ is equivalent to 
\begin{eqnarray*}\nonumber
\frac{1-b_1(S')}{f'} + \frac{e_0}{2f'} &\ge& \frac{2-2g}{f}\end{eqnarray*}
or 
\begin{eqnarray*} 
f[2-2b_1(S') +e_0]&\ge& (4-4g)f',
\end{eqnarray*}
where $f$ denotes the number of 2-simplexes in $S$. Since $f\ge f'$ the above inequality follows from 
\begin{eqnarray*}
2- 2b_1(S') +e_0 \ge 4- 4g.
\end{eqnarray*}
Since $b_1(S)=2g$ the latter inequality is equivalent to 
\begin{eqnarray}\label{target}
b_1(S') \le b_1(S) + e_0/2 -1.
\end{eqnarray}
The homological exact sequence of $(S, S')$ has the form
\begin{eqnarray*}
0\to H_2(S;\Q) &\to& H_2(S, S';\Q)\stackrel{j_\ast} \to H_1(S';\Q)\\  &\to& H_1(S;\Q) \to H_1(S, S';\Q)\to 0.\end{eqnarray*}
Here $H_2(S;\Q)=\Q$ and by the Poincar\'e duality theorem (see \cite{Hat}, Proposition 3.46) 
\begin{eqnarray}\label{duality}H_2(S,S';\Q) \simeq H^0(S-S';\Q)\end{eqnarray}
has dimension equal to the number $k$ of path-connected components of the complement $S-S'$. Formally, we find a compact deformation retract $K\subset S-S'$ such that $S-K$ deformation retracts onto $S'$ and apply Proposition 3.46 from \cite{Hat} to it; thus we obtain (\ref{duality}). 

It follows that the image of $j_\ast$ has dimension $k-1$ and therefore 
the long exact sequence implies 
\begin{eqnarray}\label{target1}
b_1(S) \ge b_1(S') - k+1.
\end{eqnarray}
Each of the connected components of the complement $S-S'$ is bounded by a simple polygonal curve having at least 3 edges. Therefore, we see that 
\begin{eqnarray}\label{three}
e_0\ge 3k
\end{eqnarray} 
and now (\ref{target}) follows from (\ref{target1}). 

Consider now the case when the surface $S$ is nonorientable, $S=N_g$. In this case the arguments are similar but we will consider the homology groups with coefficients in $\Z_2$ and the $\Z_2$-Betti numbers which we will denote $$b'_i(X)= \dim H_i(X; \Z_2).$$ Comparing $\mu(S')$ given by (\ref{sprime}) and $\mu(S)$ given by (\ref{nonorient}) and taking into account the equality
$$b_1'(S) =g,$$ we see that the inequality $\mu(S') \ge \mu(S)$ is equivalent to 
\begin{eqnarray}\label{target2}
b'_1(S') \le b'_1(S) + e_0/2 -1,
\end{eqnarray}
which is analogous to (\ref{target}). The inequality (\ref{target2}) follows from arguments similar to the ones given above with $\Z_2$ coefficients replacing the rationals $\Q$, using the Poincar\'e duality and the inequality (\ref{three}). 
\end{proof}

%
%

In the following statement we consider {\it \lq\lq small surfaces\rq\rq}, i.e. triangulated surfaces which do not depend on $n$. 
Theorems \ref{embed},  \ref{bal} and Corollaries \ref{cor13} and \ref{cor14} imply: 

\begin{corollary} One has:
\begin{enumerate}
\item If $p\ll n^{-1}$ then a random 2-complex $Y\in G(\Delta_n^{(2)}, p)$ contains\footnote{In this Corollary the word \lq\lq contains\rq\rq\, means \lq\lq contains as a simplicial subcomplex.\rq\rq} no small\footnote{In this statement one may remove the word \lq\lq small\rq\rq\, as follows from the proof of Theorem \ref{thm1intro}.}
 closed surfaces, a.a.s. 
\item If $n^{-1}\ll p\ll n^{-3/5}$ then a random 2-complex $Y$ contains small spheres but no small closed surfaces of other topological types, a.a.s.
\item If $n^{-3/5}\ll p\ll n^{-1/2}$ then a random 2-complex $Y$ contains small spheres and projective planes but no small closed surfaces of higher genera, a.a.s.
\item If $p\gg n^{-1/2}$ then a random 2-complex $Y$ contains all small spheres, projective planes, tori and Klein bottles, a.a.s.
\item If $p\gg n^{-1/2+\epsilon}$ for some $\epsilon >0$ then, given a topological type of a closed surface, there exists $f_0=f_0(\epsilon)$, such that any triangulation of the surface 
having more than $f_0$ 2-simplexes will be simplicially embeddable into a random 2-complex $Y$, a.a.s. In particular, if $p\gg n^{-1/2+\epsilon}$, 
a random 2-complex $Y$ contains small closed orientable and nonorientable surfaces of all possible topological types, a.a.s.
\end{enumerate}
\end{corollary}
\begin{proof} These statements follow from Theorem \ref{bal} and formulae (\ref{orient}) and (\ref{nonorient}). 
\end{proof}
The statement 5 of the previous Corollary can be compared with Theorem \ref{thm2intro} which deals with topological embeddings. 

\begin{corollary} For a random 2-complex $Y\in G(\Delta_n^{(2)}, p)$ with $p\gg n^{-1}$
one has \begin{eqnarray}\pi_2(Y) \not=0,\quad \mbox{and}\quad H_2(Y;\Z) \not=0\end{eqnarray} 
a.a.s. 
\end{corollary}
\begin{proof} Indeed, by the previous Corollary, for $p\ll n^{-1}$ a random 2-complex $Y$ contains a tetrahedron as a simplicial subcomplex. The fundamental class of this tetrahedron gives a nontrivial element of $H_2(Y)$. The tetrahedron can also be viewed as a sphere in $Y$ representing a nontrivial class in $\pi_2(Y)$. 
\end{proof}

The statement $H_2(Y;\Z)\not=0$ also follows from Theorem \ref{thm2} and from the result of D. Kozlov \cite{DK}.

\section{Remarks concerning the invariant $\mu(S)$} 

First we observe that $\mu(S)$ admits the following curious interpretation. 

For each vertex $u_i \in V(S)$ its degree $\deg(u_i)$ is defined as the number of edges incident to $u_i$. 
For an edge $e_i \in E(S)$ the degree $\deg(e_i)$ is defined as the number of two-dimensional simplexes incident to $e_i$. 
Next we define {\it the average vertex degree} and {\it the average edge degree} by the formulae
$$D_v(S)= v^{-1}\cdot \sum_{u_i\in V(S)} \deg (u_i), \quad D_e(S)= e^{-1}\cdot \sum_{e_i\in E(S)} \deg (e_i).$$
\begin{lemma} For any 2-complex $S$ one has
$$\mu(S)\cdot D_v(S) \cdot D_e(S)= 6.$$
\end{lemma}
\begin{proof}
The statement follows from the definition $$\mu(S)=v/f= 6\cdot \frac{v}{2e}\cdot \frac{e}{3f}$$ using the following obvious formulae
$$3f=\sum_{e_i\in E(S)}\deg(e_i), \quad 2e= \sum_{u_i\in V(S)} \deg(u_i).$$
\end{proof}

\begin{lemma} For any strongly connected 2-complex $S$ one has
\begin{eqnarray}\label{ineqarm}\mu(S) \le 1 + \frac{2}{f},\end{eqnarray}
where $f=f_S$ is the number of faces in $S$. 
\end{lemma}
\begin{proof} Without loss of generality we may assume that $S$ is pure; otherwise we apply the arguments below to the pure part of $S$.

Given a pure strongly connected 2-complex $S$, there exists a sequence of subcomplexes $T_1\subset T_2 \subset \dots \subset T_f=S$ such that 
(a) each $T_i$ has exactly $i$ faces, i.e. $f_{T_i}=i$, and (b) 
the subcomplex $T_{i+1}$ is obtained from $T_i$ by adding a single 2-simplex $\sigma_i$ with the property that the intersection $\sigma_i\cap T_i$ contains an edge of $\sigma_i$. 
If $v_i$ denotes the number of vertices of $T_i$ then $v_{i+1}\le v_i+1$. Since $v_1=3$, it follows that $v=v_f\le f+2$ implying (\ref{ineqarm}). 
\end{proof}

\begin{corollary}
Suppose that a 2-complex $S=S_1\cup S_2$ is the union of two strongly connected subcomplexes such that the intersection $S_1\cap S_2$ is at most one-dimensional.
(a) If  $S_1\cap S_2$ contains at least $4$ vertices then $\mu(S) \le 1$. (b) If the intersection $S_1\cap S_2$ contains $\ge 5$ vertices then $\mu(S) <1$. 
\end{corollary}
\begin{proof} Denote $v_i=v_{S_i}$, $f_i=f_{S_i}$, where $i=1, 2$ and, as usual, $v=v_S$, $f=f_S$.  By the previous Lemma, $v_i \le f_i+2$, and thus we obtain
\begin{eqnarray}\mu(S) &=& \frac{v_1+v_2 -v_0}{f_1+f_2} \le \frac{f_1+2 +f_2+2 - v_0}{f_1+f_2}\\
 &= &1 + \frac{4-v_0}{f_1+f_2},\end{eqnarray}
where $v_0$ is the number of vertices lying in the intersection $S_1\cap S_2$. Thus, $\mu(S)\le 1$ if $v_0\ge 4$ and $\mu(S)<1$ if $v_0>4$. 
\end{proof}

\begin{lemma} Let $S$ be a connected, pure, closed (i.e. $\partial S=\emptyset$) 2-complex with $\chi(S)=1$ having at least 3 edges of degree $\ge 3$. 
Then 
\begin{eqnarray}\label{three3}
\mu(S) \le \frac{1}{2}-\frac{1}{2f},\end{eqnarray}
where $f=f_S$ is the number of faces. 
\end{lemma}
\begin{proof} We have 
\begin{eqnarray}\label{one1}v-e+f=1
\end{eqnarray}(since $\chi(S)=1$)
 and \begin{eqnarray}\label{two2}3f\ge 2e+3.
\end{eqnarray}
The last inequality follows from the formula
$$3f= 2e +e_3+e_4+\dots$$
where $e_r$ denotes the number of edges of degree at least $r$ in $S$ with $r=3, 4, \dots$. From (\ref{one1}) and (\ref{two2}) we obtain
$v\le \frac{f}{2} - \frac{1}{2}$ implying (\ref{three3}).
\end{proof}

An example of a 2-complex satisfying the condition of the previous Lemma is the house with two rooms, see  \cite{Hat}, page 4.

\section{Topological embeddings: proof of Theorem \ref{thm2intro}}
%
\begin{proof}
We show that there exists a subdivision of $S$ which simplicially embeds into $Y$ a.a.s. 

We subdivide $S$ by introducing a new vertex in the center of each 2-simplex and connecting it to three vertices, as shown on Figure \ref{subdivide}.
\begin{figure}[h]
\begin{center}
\resizebox{9cm}{4cm}{\includegraphics[7,439][545,630]{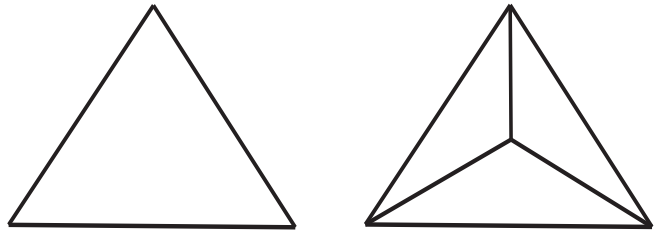}}
\end{center}
\caption{A $2$-simplex (left) and its subdivision (right).}\label{subdivide}
\end{figure}
We denote by $S'$ the new triangulation. Let $v, f$ and $v', f'$ denote the numbers of vertices and faces of $S$ and $S'$ respectively. 
Then clearly
$$v'=v+f, \quad f'=3f.$$
Therefore we find that
\begin{eqnarray}\label{sub}
\mu(S')-\frac{1}{2} = \frac{1}{3}\left(\mu(S) -\frac{1}{2}\right).
\end{eqnarray}
We claim that a similar formula holds for $\tilde \mu$, i.e. 
\begin{eqnarray}\label{subtilda}
\tilde \mu(S')-\frac{1}{2} = \frac{1}{3}\left(\tilde \mu(S) -\frac{1}{2}\right).
\end{eqnarray}
Indeed, let $T\subset S$ be a subcomplex. Then its subdivision $T'$ (defined as explained above) is a subcomplex of $S'$, and the numbers $\mu(T)$ and $\mu(T')$ are related by the equation (\ref{sub}). We show below that
\begin{eqnarray}\label{enough}
\tilde \mu(S') = \min_{T\subset S} \mu(T').
\end{eqnarray}
Clearly, (\ref{enough}) implies 
$$
\tilde \mu(S') = \min_{T\subset S} \left[\frac{1}{3} \left(\mu(T)-\frac{1}{2}\right) +\frac{1}{2}\right]= \frac{1}{3}\left(\tilde \mu(S) -\frac{1}{2}\right) +\frac{1}{2}
$$
which is equivalent to (\ref{subtilda}). 

To prove the formula (\ref{enough}) consider a subcomplex $R\subset S'$. Each 2-simplex $\sigma$ of $S$ determines three 2-simplexes of $S'$ which we denote by 
$\sigma_1, \sigma_2, \sigma_3$. We want to show that we may replace $R$ by a subcomplex $R_1\subset S'$ such that $\mu(R_1)\le \mu(R)$ and 
either $R_1$ contains all simplexes $\sigma_1, \sigma_2, \sigma_3$ or it contains none of them. 

Suppose that $R$ contains $\sigma_1$ and $\sigma_2$ but does not contain $\sigma_3$. Then $R_1=R\cup\sigma_3$ has the same number of vertices and greater number of faces, i.e. 
$\mu(R_1)<\mu(R)$. 

Suppose now that $R$ contains only one simplex among the $\sigma_i$'s; assume, that, say, $\sigma_1\subset R$ and $\sigma_2\not\subset R$ and 
$\sigma_3\not\subset R$. (A) If $\mu(R) \ge 1/2$, define $R_1$ by $R_1=R\cup \sigma_2\cup \sigma_3$. Then $\mu(R_1) \le \mu(R)$. 
(B) If $\mu(R)\le 1$ define $R_1$ as $R$ with $\sigma_1$ removed; then $\mu(R_1) \le \mu(R)$. Clearly at least one of the cases (A) or (B) holds and we proceed by induction, repeating this procedure with respect to all 2-simplexes $\sigma\subset S$. Thus we see that the minimum in 
$$\tilde \mu(S')= \min_{R\subset S'} \mu(R)$$ is achieved on subcomplexes $R\subset S'$ which have the form $R=T'$ for some $T\subset S$.  
This completes the proof of (\ref{subtilda}).

For $r=0,1, 2, \dots$ denote by $S^{r}$ the simplicial 2-complex which is obtained from $S$ by $r$ consecutive subdivisions as above. Then from (\ref{subtilda}) we obtain
\begin{eqnarray}\label{subtildar}
\tilde \mu(S^r)-\frac{1}{2} = \frac{1}{3^r}\left(\tilde \mu(S) -\frac{1}{2}\right).
\end{eqnarray}
We see that this sequence approaches $1/2$ as $r\to \infty$. It follows that, given $\epsilon >0$, for all sufficiently large $r$ we have 
$$\tilde \mu(S^r) \ge 1/2 -\epsilon.$$
Thus, the assumption $p\gg n^{-1/2+\epsilon}$ implies $p\gg n^{-\tilde\mu(S^r)}$ and now we may apply Theorem \ref{embed} to conclude that the $r$-th subdivision $S^r$ simplicially embeds into $Y$, a.a.s.
Hence we see that $S$ topologically embeds into $Y$, a.a.s.
\end{proof}
\begin{remark}{\rm 
The result of Theorem \ref{thm2intro} cannot be improved (without adding extra hypothesis) despite a special type of subdivision used in the proof. Indeed, one sees from formulae (\ref{orient}) and (\ref{nonorient}) and Theorem \ref{bal} that for a closed orientable surface $\Sigma_g$ of genus $g\ge 1$ one has $\tilde\mu(\Sigma_g)\to 1/2$ as the number of 2-simplexes $f$ goes to infinity. A similar conclusion is valid for nonorientable surfaces $N_g$ with $g\ge 2$. 
}
\end{remark}
\begin{remark}{\rm Consider the following invariant $\sign(X)\in \{+1, -1, 0\}$ of a simplicial 2-complex:
$$\sign(X) = \sign\left(\tilde\mu(X)-\frac{1}{2}\right).$$
Formula (\ref{sub}) seems to suggest that it is topologically invariant. However in (\ref{sub}) we used a special type of subdivisions. The following example shows that {\it in general 
$\sign(X)$ is not topologically invariant.} Consider the 2-complex $X$ shown in Figure \ref{triod} (left) which is the union of three triangles having a common edge. 
Let $Y_k$ be obtained by adding $k$ new vertices along the common edge and connecting them to the remaining vertices, see Figure \ref{triod} (right). 
\begin{figure}[h]
\begin{center}
\resizebox{10cm}{4cm}{\includegraphics[17,348][551,579]{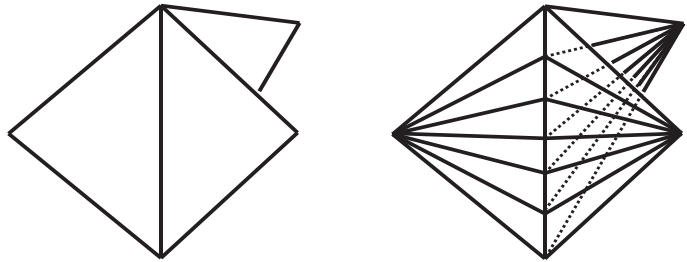}}
\end{center}
\caption{Complex $X$ (left) and its subdivision $Y_k$ (right).}\label{triod}
\end{figure}
One has $\tilde \mu(X)= 5/3$ and therefore $\sign(X)=+1$. However 
$$\tilde \mu(Y_k) \le \mu(Y_k) = \frac{k+5}{3k+3}.$$ Thus, for $k>7$, one has $\tilde\mu(Y_k)<1/2$ and $\sign(Y_k)=-1$. 
}
\end{remark}

\newcommand{\arxiv}[1]{{\texttt{\href{http://arxiv.org/abs/#1}{{arXiv:#1}}}}}

\newcommand{\MRh}[1]{\href{http://www.ams.org/mathscinet-getitem?mr=#1}{MR#1}}

\vskip 0.5cm

 Armindo Costa

Department of Mathematical Sciences

Durham University

Durham, DH1 3LE, UK

a.e.costa@durham.ac.uk

\vskip 0.5cm

Michael Farber

Department of Mathematical Sciences

Durham University

Durham, DH1 3LE, UK

Michael.farber@durham.ac.uk

http://maths.dur.ac.uk/$\sim$dma0mf/

\vskip 0.5cm

Thomas Kappeler

Mathematical Insitutte 

University of Zurich

Winterthurerstrasse 190, CH-8057 

Zurich, Switzerland

thomas.kappeler@math.uzh.ch



\end{document}